\documentclass{amsart}
\usepackage{amsmath}
\usepackage{mathtools}
\usepackage{graphicx}
\usepackage[colorlinks=true, allcolors=blue]{hyperref}
\usepackage{comment}
\usepackage{amsfonts}
\usepackage{amsthm}
\usepackage{newlfont}
\usepackage{amscd}
\usepackage{amsgen}
\usepackage{amssymb}
\usepackage{mathrsfs}	
\usepackage{longtable}
\usepackage{listings}
\usepackage{extarrows}
\usepackage{tikz}
\usepackage{tikz-cd}
\usepackage{verbatim}
\numberwithin{equation}{section}
\usepackage[all]{xy}
\usepackage{color}
\usepackage{amssymb}
\usepackage{geometry}
\usepackage{tikz-cd}
\usepackage{mathtools}
\usetikzlibrary{matrix,shapes,arrows,decorations.pathmorphing}
\usepackage{calligra}
\usepackage{mathrsfs}

\let\blb\mathbb
\def\CC{{\blb C}}

\def \PP{{\blb P}}

\def \ZZ{{\blb Z}}

\def \LL{{\blb L}}
\def \OO{{\blb O}}

\let\cal\mathcal

\def\Cc{{\cal C}}

\def\Ec{{\cal E}}
\def\Fc{{\cal F}}
\def\Gc{{\cal G}}

\def\Ic{{\cal I}}

\def\Lc{{\cal L}}
\def\Mc{{\cal M}}
\def\Nc{{\cal N}}
\def\Oc{{\cal O}}
\def\Pc{{\cal P}}
\def\Qc{{\cal Q}}

\def\Sc{{\cal S}}

\def\Uc{{\cal U}}

\def\Xc{{\cal X}}

\newtheorem{lemma}{Lemma}[section]
\newtheorem{proposition}[lemma]{Proposition}
\newtheorem{theorem}[lemma]{Theorem}
\newtheorem{corollary}[lemma]{Corollary}

\newtheorem{definition}[lemma]{Definition}
\newtheorem{conjecture}[lemma]{Conjecture}

\theoremstyle{remark}

\newtheorem{remark}[lemma]{Remark}

\def\dbcoh{D^b\operatorname{Coh}}
\def\wt{\widetilde}
\def\arw{\longrightarrow}
\def\Hom{\operatorname{Hom}}
\def\Ext{\operatorname{Ext}}
\def\Aut{\operatorname{Aut}}

\def\deg{\operatorname{deg}}
\def\proj{\operatorname{Proj}}
\def\sym{\operatorname{Sym}}

\DeclareMathOperator{\sHom}{\mathscr{H}\text{\kern -3pt {\calligra\large om}}\,}

\makeatletter
\newcommand{\customlabel}[2]{%
   \protected@write \@auxout {}{\string \newlabel {#1}{{#2}{\thepage}{#2}{#1}{}} }%
   \hypertarget{#1}{#2}
}
\makeatother

\title{Mukai duality via roofs of projective bundles}
\author{Micha\l\ Kapustka and Marco Rampazzo}
\begin{document}
\maketitle

\begin{abstract}
We investigate a construction providing pairs of Calabi--Yau varieties described as zero loci of pushforwards of a hyperplane section on a roof as described in \cite{kanemitsu}. We discuss the implications of such construction at the level of Hodge equivalence, derived equivalence and $\LL$-equivalence. For the case of K3 surfaces, we provide alternative interpretations for the Fourier--Mukai duality in the family of K3 surfaces of degree 12 of \cite{MukaiK3}. In all these constructions the derived equivalence lifts to an equivalence of matrix factorizations categories.
\end{abstract}

\section{Introduction}\label{introduction}

The construction of non-isomorphic but derived equivalent pairs of varieties with vanishing first Chern class has been object of a recent flurry of articles (\cite{imou}, \cite{or}, \cite{bcp}, \cite{hl}, \cite{kr}) originated from the Pfaffian--Grassmannian equivalence \cite{rodland} and developed alongside with the notion of phase transition of a gauged linear sigma model, a physical phenomenon connecting separate theories via a process which has been mathematically described by means of variation of GIT.\\
\\
From a merely mathematical perspective, the existence of so-called \textit{multiple geometric phases} led to the construction of several instances of derived equivalence between non isomorphic varieties, while the geometric description of many of such pairs allowed the establishment of $\LL$-equivalence, which is a relation in the Grothendieck ring of varieties given by the difference of the classes of such varieties annihilating a power of the class of the affine line, such as in \cite{imou}. The interplay between derived equivalence and $\LL$-equivalence has been object of conjectures as described in \cite{ks}.\\
\\
A recurring pattern emerges from some particularly symmetric examples of derived equivalence and $\LL$-equivalence constructions where the pair is defined by the two pushforwards of a hyperplane section of a smooth Fano variety given by an incidence correspondence, along its two natural surjections. An interesting natural setup arises from some Fano varieties investigated by Kanemitsu in \cite{kanemitsu}, called \emph{roofs}. In this class lie, for example, the constructions of \cite{imou}, \cite{kr}. Moreover, these roof diagrams are a natural setup  for testing the DK conjecture of \cite{bondalorlovdk}, \cite{kawamatadk}.\\
\\
In this paper, we investigate constructions of pairs of Calabi--Yau varieties emerging from roofs from the point of view of Hodge equivalence, derived equivalence and $\LL$-equivalence. We first describe a Hodge isometry at the level of middle cohomology which, for example in the case of K3 surfaces, permits us to prove that the surfaces are derived equivalent but not isomorphic. $\LL$-equivalence for the related Calabi--Yau pair is easily proven in all known cases, however a general proof is missing.\\
\\
As a working  example of the above construction we investigate the non-homogeneous case given by the pair of K3 surfaces cut out by sections of a twisted Ottaviani bundle on a five dimensional quadric in $\PP^6$: the sections describing the surfaces are pushforwards of a hyperplane section of the projectivization of such an Ottaviani bundle. By means of this construction, we provide a description of the general K3 surface of degree 12 projected to $\PP^6$, and we give an alternative formulation of the self-duality of the family of K3 surfaces studied already by Mukai in \cite{MukaiK3} and later in \cite{imouk3} and \cite{hl}. This answers a question posed in \cite[Rem.4.2]{imouk3}. 
\\
\\
The paper is organized as follows. In Section \ref{roofs} we recall the definition of roofs of $\PP^{r-1}$-bundles given by Kanemitsu in \cite{kanemitsu}. Therefore we present a construction to associate a pair of Calabi--Yau manifolds to a hyperplane section in a roof $X$, and describe how such pairs should be $\LL$-equivalent.
In Section \ref{hodge}, given a roof $X$ and a hyperplane $M$ giving rise to a pair $Y, \wt Y$ of Calabi--Yau varieties, we construct two direct sum decompositions of the middle cohomology of $M$ containing the middle cohomology of respectively $Y$ and $\wt Y$.
Specializing this picture to pairs of K3 surfaces, we prove that the two direct sum decompositions of the middle cohomology of $M$ provide a Hodge isometry between the transcendental lattices of the K3 surfaces. Hence, pairs of K3 surfaces related by a roof are derived equivalent by the derived global Torelli theorem. 
In Section \ref{k3s} we introduce the pairs of K3 surfaces of degree 12 in $\PP^6$ and prove that they are pairs of general K3 surfaces of degree 12. Furthermore, we show that the general pair of K3 surfaces related by a roof is not isomorphic. We prove it by describing the action of the Hodge isometry on the discriminant group of the transcendental lattice of such K3 surfaces. We study the associated Fourier--Mukai kernel in Section \ref{kernels}, in light of the conjecture of Kuznetsov and Shinder. In Section \ref{branes} we observe that the obtained derived equivalences lift to equivalences of matrix factorization categories by means of an application of Kn\"orrer periodicity. 

\subsection*{Acknowledgments}
We are thankful to Atsushi Ito, Makoto Miura, Shinnosuke Okawa, Kazushi Ueda and Akihiro Kanemitsu for giving interesting comments and corrections to the first draft.
Part of the work has been developed while the second author was at Institut de Mathématiques de Toulouse with Laurent Manivel, which also participated in the beginning of the project, we would like to express our gratitude to him for many valuable insights.
We would also like to thank Claudio Onorati for clarifying discussions, and the anonymous referee for the very careful reading and important comments. The first author is supported by the Polish National Science Center project number 2018/31/B/ST1/02857. 
The second author is supported by the PhD program at the University of Stavanger.


\section{Roofs and Calabi--Yau pairs}\label{roofs}
In this paper we work over the field of complex numbers. We shall use the notation $\PP(\Ec):=\proj(\sym \Ec^\vee)$, whenever $\Ec$ is a vector bundle or a vector space.
\begin{definition}\cite[Definition 0.1]{kanemitsu}
A simple Mukai pair $(\Ec, B)$ of dimension $n$ and rank $r$ is the data of a Fano $n$-fold $B$ of Picard number one  and an ample vector bundle $\Ec$ of rank $r$ such that:
\begin{enumerate}
    \item $c_1(\Ec)=c_1(B)$
    \item $\PP(\Ec^\vee)$ admits a second $\PP^{r-1}$-bundle structure.
\end{enumerate}
\end{definition}
\begin{definition}\cite[Definition 0.1]{kanemitsu}
A roof of dimension $n+r-1$ is a Fano manifold $X$  isomorphic to the projectivization of a simple Mukai pair $(\Ec, B)$ of dimension $n$ and rank $r$.
\end{definition}
A Fano manifold $X$ is a roof if and only if it has Picard number two, it admits two different  $\PP^{r-1}$-bundle structures and it has index $r$. Moreover, in this case there exists a line bundle $\Lc$ on $X$ which restricts to $\Oc(1)$ on all the fibers of both the $\PP^{r-1}$-fibrations \cite[Proposition 1.5]{kanemitsu}.\\

\begin{lemma}\label{ksi}
Let $\xi$ be the class of $\mathcal L$ in $H^2(X,\mathbb Z)$. Let furthermore $h_1$ and $h_2$ be respective pullbacks of the hyperplane classes on the bases $B_1$ and $B_2$. Then there exists $k\in \mathbb{Z}$ such that $k\xi=h_1+h_2$.
\end{lemma}
\begin{proof} Note that by construction both the pairs $(\xi, h_1)$ and $(\xi,h_2)$ generate the Picard group of $X$. If we hence write $h_2=a\xi +b h_1$ with $a,b\in \mathbb{Z}$ then $b=\pm 1$ by the fact $(\xi,h_2)$ generates $Pic(X)$. Since $\xi$ meets general fibers of both fibrations in one point we have $b=-1$.
\end{proof}

\begin{remark}
Note that in all known cases  Lemma \ref{ksi} is true with $k=1$.
\end{remark}
Let us now introduce the following construction:
\begin{definition}\label{CYpair}
Let $X$ be a roof of dimension $n+r-1$ and let $(\Ec, B)$, $(\wt \Ec, \wt B)$ be two Mukai pairs such that $X\simeq \PP(\Ec^\vee)\simeq\PP(\wt\Ec^\vee)$. We call Calabi--Yau pair associated to the roof a pair $(Y, \wt Y)$ of Calabi--Yau $n-r$--folds such that there exists $S\in H^0(X, \mathcal{L})$ yielding $Y=Z(\pi_*S)$ and $\wt Y=Z(\wt\pi_*S)$, where $\pi:X\arw B$ and $\wt\pi:X\arw\wt B$ are the projective bundle maps.
\end{definition}
The data of such Calabi--Yau pair can be summarized by the following diagram:
\begin{equation}\label{generalpair}
\begin{tikzcd}
 & p^{-1}(Y)\arrow[hook]{r}{j}\arrow[swap]{lddd}{p|_{p^{-1}(Y)}}& M\arrow[hook]{dd}\arrow{lddd}[swap]{p}\arrow{rddd}{\wt p}\arrow[hookleftarrow]{r}{\wt j}& \wt p^{-1}(\wt Y)\arrow{rddd}{\wt p|_{\wt p^{-1}(\wt Y)}} & \\
 & & & & \\
 & & X\arrow{ld}{\pi} \arrow{rd}[swap]{\wt\pi} & & \\
Y\arrow[hookrightarrow]{r} & B & & \wt B\arrow[hookleftarrow]{r} & \wt Y\\
\end{tikzcd}
\end{equation}
where we denote by $M$ the zero locus of $S$.\\
Roofs have been partially classified by Kanemitsu \cite[Theorem 5.12]{kanemitsu}. In particular, the classification is complete for roofs yielding Calabi--Yau pairs of dimension $d\leq 2$, for roofs of dimension $n+r-1\leq 7$ and for roofs of $\PP^1$-bundles.\\
Except for one specific case that will be discussed later, all known roofs are homogeneous varieties. In fact, they are classified by the Lie algebra type of the automorphism group.\\
\\
It is a natural question to ask whether every roof provides pairs of derived equivalent Calabi--Yau manifolds. Such conjecture, which we state here below, is supported by several worked examples, despite the lack of a general proof.
\begin{conjecture}\label{conjecture}
Let $X$ be a roof as in Diagram $\ref{generalpair}$. Then there exists a derived equivalence $\dbcoh(Y)\simeq\dbcoh(\wt Y)$, where $(Y,\wt Y)$ is a Calabi--Yau pair associated to $X$ in the sense of Definition $\ref{CYpair}$.
\end{conjecture}
\begin{remark}
The DK conjecture states that if two smooth projective varieties are related by a flop, they are derived equivalent \cite[Conjecture 1.2]{kawamatadk}, \cite[Conjecture 4.4]{bondalorlovdk}. In this context, Conjecture \ref{conjecture} is particularly interesting if we observe that the total spaces $\Ec^\vee$ and $\wt\Ec^\vee$ are related by a flop. A positive answer to  the DK conjecture for such flops has been given for the roof of type $G_2$ by \cite{uedaflop}, and for the roofs of type $C_2$ and $A_4$ by \cite{morimuraflop}, but again a general proof of the validity of the DK conjecture for bundles related by a roof is missing.\\
The problem of finding a derived equivalence for the total spaces is strictly related to proving that the Calabi--Yau zero loci are derived equivalent: in fact, one has the following diagram:
\begin{equation*}
    \begin{tikzcd}
            &     \mathcal L^{\vee}  
            \ar[swap]{ld}{f}\ar{rd}{g}    
            &      \\
    \Ec^\vee&  X\ar{ld}\ar{rd}\ar[hook]{u}         &\wt\Ec^\vee\\
   \hspace{-20pt}Y\subset B\ar[hook]{u}      &     &\wt B\ar[hook]{u}\supset \wt Y\hspace{-20pt}
    \end{tikzcd}
\end{equation*}
where $f$ and $g$ are blowups of respectively $\Ec^\vee$ in $B$ and $\wt\Ec^\vee$ in $\wt B$, the bases are embedded in the total spaces as zero sections. Then it is possible to write the derived category of the total space of $\mathcal L^{\vee}$
in two ways, each of them being a semiorthogonal decomposition containing a twist of $\dbcoh(\Ec^\vee)$ and $r-1$ twists of $\dbcoh (B)$, or a totally similar decomposition on the other side of the diagram. This picture mirrors the one given by Diagram \ref{generalpair}: in fact, in the existing worked examples, the strategy of the proof adopted for the total spaces is the same that has been used for the zero loci (see for example the relation of \cite{morimuraflop} with \cite{kr} and of \cite{uedaflop} with \cite{Kuznetsov}).
\end{remark}

\subsection{\texorpdfstring{$\LL$}{L}-equivalence in the Grothendieck ring of varieties}

We observe that the fibers of the surjection $p$ have the following description:
\begin{equation*}
p^{-1}(y) \simeq \left\{
\begin{array}{ll}
\PP^{r-1} & \operatorname{if} y\in Y\\
\PP^{r-2} &\operatorname{if} y\in B\setminus Y
\end{array}
\right.
\end{equation*}
due to the fact that $Y$ is the zero locus of the pushforward of $S$ to $B$. Clearly, the same holds for $\wt p$, $\wt B$ and $\wt Y$.\\
\\
Let us consider the roof from the point of view of the Grothendieck ring of varieties. First of all observe that the bases $B$ and $\wt B$ are stably birational i.e. by results of \cite[Section 1.7]{LarsenLunts} they have equal class in the quotient of the Grothendieck ring by the ideal generated by the $\mathbb{L}$. 

However, if we assume the stronger condition that $B$ and $\wt B$ have equal class in the Grothendieck ring we get an interesting consequence, namely the difference of the classes of a pair of Calabi--Yau varieties associated to the roof  annihilates the $r-1$-th power of the class of the affine line.

Indeed, the map $p$ is a piecewise trivial fibration. This is a consequence of the fact that $p^{-1}(Y) = \PP(\Ec^\vee|_Y)$, and that $p^{-1}(B\setminus Y)= \PP((\Ec|_{B\setminus Y}/(\pi_*S|_{B\setminus Y}\otimes\Oc_B|_{B\setminus Y}))^\vee)$, the same holds for $\wt\pi$.\\
Now, let $M$ be the hyperplane section of $X$ defining the pair. Then we have the following two descriptions of the class of $M$:
$$
[M] = [\PP^{r-1}]\cdot [Y]+[\PP^{r-2}]\cdot [B\setminus Y] 
$$
$$
[M] = [\PP^{r-1}]\cdot [\wt Y]+[\PP^{r-2}]\cdot [\wt B\setminus \wt Y] 
$$
By the relations defining the Grothendieck ring of varieties we have $[B\setminus Y] = [B]-[Y]$, and the same holds for the second equation. Then, subtracting the two equations above, we get
\begin{equation*}
\begin{split}
0&=  [\PP^{r-1}]\cdot [Y]+[\PP^{r-2}]\cdot( [B]-[Y] ) - [\PP^{r-1}]\cdot [\wt Y] - [\PP^{r-2}]\cdot ([\wt B]-[\wt Y] ) \\
&= ([\PP^{r-1}]-[\PP^{r-2}])\cdot([Y]-[\wt Y]) + [\PP^{r-2}]\cdot([B]-[\wt B]).
\end{split}
\end{equation*}
Given the identity
$$
[\PP^k] = 1+\LL + \LL^2+\cdots\LL^k,
$$
we have:
\begin{equation}\label{lequivalencegeneral}
([Y]-[\wt Y])\cdot\LL^{r-1} + [\PP^{r-2}]\cdot([B]-[\wt B])=0.
\end{equation}
\begin{remark}
In most of the known examples of roofs, Equation \ref{lequivalencegeneral} provides $\LL$-equivalence for the associated Calabi--Yau pairs because the bases $B$, $\wt B$ of the roof are isomorphic. This is the case of roofs of type $A^M_k$, $A^G_{2k}$, $D_k$ and $G_2^\dagger$. For the roof of type $G_2$, it has been proved that $[B]=[\wt B]=[\PP^5]$ in  \cite[Proposition 2.2]{imou}, and the $\LL$-equivalence follows. Note that for any roof we must have  $([B]-[\wt B])\cdot[\PP^{r-1}] =0$, which suggests $[B]-[\wt B]=0$. However, although we are not aware of any counterexample, we are unable to prove this fact.
\end{remark}


\section{Hodge structures}\label{hodge}
In this section we generalize the blowup formula for cohomology in order to compare Hodge structures in the algebraic middle cohomologies of Calabi--Yau pairs associated to a roof.

\subsection{Cohomological Cayley trick}
Let $X$ be a roof such as in Diagram \ref{generalpair} and let $(Y, \wt Y)$ be a Calabi--Yau pair associated to a smooth hyperplane section $M$. The scope of this section is to give a dual description of the cohomology of $M$ in terms, respectively, of the cohomologies of $Y$ and $B$ and the cohomologies of $\wt Y$ and $\wt B$.\\
\\
More generally, given a Mukai pair $(\Ec, B)$ and a hyperplane section $M\subset X\simeq \PP(\Ec^\vee)$, we establish the following diagram, which will be the setting of Theorem \ref{integralcohomology}:
\begin{equation*}
\begin{tikzcd}[row sep = large, column sep = large]
p^{-1}(Y)\arrow{d}{q}\arrow[hook]{r}{j} & M \arrow[hook]{r}\arrow[swap]{d}{p} & X\arrow{dl}{\pi}\\
Y\arrow[hook]{r}& B					&
\end{tikzcd}
\end{equation*}
where we assume $Y = Z(\pi_*S)$ smooth, $S$ is the section cutting out $M$ and $p$ is a fibration such that:
\begin{equation*}
p^{-1}(x) \simeq \left\{
\begin{array}{ll} \mathbb P^{r-1}& x\in Y\\\mathbb P^{r-2} & x\in B\setminus Y
\end{array}
\right.
\end{equation*}
and $q$ is the restriction of $p$ to the preimage $p^{-1}(Y)$.\\
\\
\begin{theorem}[Cohomological Cayley trick]\label{integralcohomology}
Let $B$ be a K\"ahler manifold of  dimension $n$, let $\Ec$ be a vector bundle of rank $r$ over $B$ and $X=\PP(\Ec^\vee)$. Then, given a smooth hyperplane section $M=Z(S)\subset X$ and a smooth variety $Y = Z(\pi_{*}S)\subset B$ of codimension $r$, there exists an isomorphism of Hodge structures
\begin{equation}\label{decomposition}
\begin{tikzcd}
\displaystyle \Xi: \bigoplus_{i=0}^{r-2}H^{k-2i}(B, \mathbb Z) \oplus H^{k-2r+2}(Y,\mathbb Z) \arrow{r}{\sim} & H^{k}(M,\mathbb Z)
\end{tikzcd}
\end{equation}
which acts on classes in the following way:
\begin{equation*}
\begin{tikzcd}
(x_{0},\dots x_{r-2},z)\arrow{r} & p^*x_{0} + p^*x_{1} \cup \xi +\dots +p^{*}x_{r-2}\cup\xi^{r-2}+ j_*q^*z
\end{tikzcd}
\end{equation*}
where $\xi\in H^2(M,\ZZ)$ is the restriction to $M$ of the hyperplane class of $X$ related to the Grothendieck line bundle $\Oc_X(1)$.
\end{theorem}
\begin{proof} The theorem is part of mathematical folklore. Its proof is analogous to the proof of the blowup formula \cite[Theorem 7.31]{voisin} and now contained in the recent paper \cite[Proposition 48]{BFM}.
\end{proof}

\subsection{The middle cohomology}
Let us specialize to $k = n+r-2$. Then Theorem \ref{integralcohomology} gives the following morphism of middle cohomologies:
\begin{equation}\label{middlecohomologyembedding}
\begin{tikzcd}[column sep = large]
H^{n-r}(Y,\ZZ)\ar[hook]{rrr}{\Xi|_{H^{n-r}(Y,\ZZ)}=j_{*}\circ q^{*}} &&& H^{n+r-2}(M,\ZZ) 
\end{tikzcd}
\end{equation}

Note also that 
\begin{equation}\label{perp trans}  \Xi\left( \bigoplus_{i=0}^{r-2}H^{n+r-2-2i}(B,\ZZ)\right) \perp  j_{*}\circ q^{*}( H^{n-r}(Y,\ZZ))
\end{equation} where $\perp$ is taken with respect to the cup product in $H^{n+r-2}(M,\ZZ)$. Indeed, this follows from dimensional reasons since $x_i|_Y\cup H^{n-r}(Y,\ZZ)=0$ for $x_i\in H^{n+r-2-2i}(B,\ZZ)$ where $i\leq r-2$.\\
\\
Furthermore, we claim that $j_{*}\circ q^{*}$ preserves the cup product pairing up to a sign determined by the rank of $\Ec$.
\begin{lemma}\label{polarized}
For every $D_1, D_2\in H^{n-r}(Y, \ZZ)$, the map $j_{*}\circ q^{*}$ of Equation \emph{\ref{middlecohomologyembedding}} satisfies the following identity, where $r$ is the rank of $\Ec$:
\begin{equation}\label{hodgeisometrywithsign}
    (D_1\cdot D_2)_Y =(-1)^{r-1} (j_*  q^* D_1 \cdot j_*  q^* D_2)_M.
\end{equation}
\end{lemma}
\begin{proof}
Let us work on the right-hand side of Equation \ref{hodgeisometrywithsign}: by an application of the projection formula we have
$$
(j_* q^* D_1\cdot j_* q^* D_2)_M = j_*(j^*j_* q^* D_1\cdot  q^* D_2)_{p^{-1}(Y)}.
$$
Let us focus on the term $j^*j_* q^* D_1$. By the self-intersection formula (\cite[Theorem 1]{LascuMumfordScott}) we have:
$$j^*j_* q^* D_1=q^* D_1\cdot c_{r-1} (\Nc_{p^{-1}(Y)|M}).$$

Substituting this in the main equation we get
\begin{equation}\label{intproduct}
    \begin{split}
        (j_* q^* D_1\cdot j_* q^* D_2)_M &= j_*( q^* D_1\cdot c_{r-1} (\Nc_{p^{-1}(Y)|M})\cdot  q^* D_2)_{p^{-1}(Y)}.\\
    \end{split}
\end{equation}

The class $c_{r-1} (\Nc_{p^{-1}(Y)|M})\in H^{2r-2}(M; \ZZ)$ can be described in terms of the class $\xi$ of the restriction to $M$ of the Grothendieck line bundle $\Oc_{\PP(\Ec^\vee)}(1)$ and the generators $C_i$ of the cohomology ring of $B$ as

$$
c_{r-1} (\Nc_{p^{-1}(Y)|M})= a\xi^{r-1} + \sum b_i \xi^i C_i \in H^{2r-2}(M; \ZZ)
$$

\noindent with $a, b_i\in \ZZ$. Since the only contributing term of $c_{r-1} (\Nc_{p^{-1}(Y)|M})$ in Equation \ref{intproduct} is $a\xi^{r-1}$, the proof reduces to showing that $a=(-1)^{r-1}$. This can be done observing that 
$$
a =\deg c_{r-1} (\Nc_{p^{-1}(Y)|M}|_{F})
$$
where $F$ is the fiber of $M$ over a point in $Y$, and it is isomorphic to $\PP^{r-1}$. By the following sequence of normal bundles
$$
0\arw \Nc_{p^{-1}(Y)|M}\arw \Nc_{p^{-1}(Y)|X}\arw \Oc_{X}|_{p^{-1}(Y)}(1)\arw 0
$$
and by the fact that the restriction of $\Nc_{p^{-1}(Y)|X}$ to $F$ is trivial, we get $a = 1$ if $r$ is odd, and $a = -1$ otherwise.
\end{proof}
Suppose  now that $H^*(B, \ZZ)$ is algebraic (which holds for example for rational homogeneous varieties), then the only non-algebraic part of the middle cohomology of $M$ comes from $H^{n-r}(Y,\ZZ)$. More precisely,  since  $j_{*}\circ q^{*}$ and $p^*$ map algebraic classes to algebraic classes,  we have
$$
H^{n+r-2}_{alg}(M,\ZZ) =  \bigoplus_{i=0}^{r-2}H^{n+r-2-2i}(B,\ZZ) \oplus H^{n-r}_{alg}(Y,\ZZ).
$$
Indeed, the sum of algebraic classes in the right hand side is algebraic, but also whenever we take an algebraic class in $M$ and decompose it by means of Equation \ref{decomposition}, since all its components from $\bigoplus_{i=0}^{r-2}H^{n+r-2-2i}(B,\ZZ)$ are algebraic by assumption then also the component in $H^{n-r}(Y,\ZZ)$ must be algebraic.

Moreover, if we define $T_Y := H^{n-r}_{alg}(Y,\ZZ)^{\perp}\subset H^{n-r}(Y,\ZZ)$ and 
$T_M := H^{n+r-2}_{alg}(M,\ZZ)^{\perp}\subset H^{n+r-2}(M,\ZZ)$, by Equation \ref{perp trans} we also have $j_{*}\circ q^{*} (T_Y)=T_M$ and by Lemma \ref{polarized} $ j_{*}\circ q^{*}$ defines an isometry between these lattices up to a sign depending on the rank of $\Ec$.
The results of Theorem \ref{integralcohomology} and Lemma \ref{polarized} apply for each side of a roof diagram like Diagram \ref{generalpair}. Hence, we have two maps $\Xi$ and $\wt{\Xi}$. Now, provided that both $B$ and $\wt B$ have algebraic cohomology, we have the following isomorphism of integral Hodge structures:
\begin{equation}\label{hodgeisometry}
    \begin{tikzcd}[column sep = scriptsize]
   \wt \Xi^{-1}\circ \Xi:   \bigoplus_{i=0}^{r-2}H^{k-2i}(B, \mathbb Z) \oplus H^{k-2r+2}(Y,\mathbb Z) \ar{rr}{\sim} & & \bigoplus_{i=0}^{r-2}H^{k-2i}(\wt B, \mathbb Z) \oplus H^{k-2r+2}(\wt Y,\mathbb Z) 
    \end{tikzcd}
\end{equation}
which by Theorem \ref{integralcohomology}, Lemma \ref{polarized} and Equation \ref{perp trans}, after restriction to $T_Y$ determines a Hodge isometry $T_Y\simeq T_{\widetilde Y}$ defined by $(\wt j_{*} \circ {\wt q^{*}|_{T_{{\wt Y}}})^{-1}}\circ j_{*} \circ q^{*}|_{T_Y}$. In light of the derived global Torelli theorem for K3 surfaces \cite[Theorem 3.3]{orlov}, this Hodge isometry gives us information about the derived categories of pairs associated to a roof with $n-r=2$.\\

\begin{corollary}\label{derivedequivalenceK3}
Let $\Ec$ and $\wt\Ec$ be vector bundles of rank $r$ on rational homogeneous bases $B$ and $\wt B$. Let $X\simeq \PP(\Ec^\vee)\simeq \PP(\wt\Ec^\vee)$ be a roof of dimension $2r+1$ and $(Y, \wt Y)$ the Calabi--Yau pair associated to $X$ defined by a smooth section $S\in H^0(X, \Lc)$. Then $Y$ and $\wt Y$ are derived equivalent.
\end{corollary}
\begin{proof}
Given the dimension and rank of the Mukai pairs, $Y$ and $\wt Y$ are K3 surfaces. By the fact that $B$ and $\wt B$ are rational homogeneous, Equation \ref{hodgeisometry} provides an isometry of transcendental lattices $T_Y\simeq T_{\wt Y}$. This, in turn, by the derived global Torelli theorem \cite[Theorem 3.3]{orlov}, proves that $Y$ and $\wt Y$ are derived equivalent.
\end{proof}

\begin{proposition}\label{oddcohomology}
Let $X\simeq \PP(\Ec^\vee)\simeq \PP(\wt\Ec^\vee)$ be a roof of dimension $2r+2k$, where $\Ec$ and $\wt\Ec$ are vector bundles of rank $r$ such that their bases $B$ and $\wt B$ have no odd-degree integral cohomology. Let $(Y, \wt Y)$ be the Calabi--Yau pair associated to $X$ defined by a section of $\Lc$ with smooth zero locus $M\subset X$. Then there exists a Hodge isometry $H^{2k+1}(Y,\ZZ)\simeq H^{2k+1}(\wt Y,\ZZ)$.
\end{proposition}
\begin{proof}
In this setting, Theorem \ref{integralcohomology} defines the following isomorphism:
\begin{equation*}\label{decomposition2}
\begin{tikzcd}
\bigoplus_{i=0}^{r-2}H^{2r+2k-1-2i}(B, \mathbb Z) \oplus H^{2k+1}(Y, \mathbb Z) \arrow{r}{\sim} & H^{2r+2k-1}(M, \mathbb Z)
\end{tikzcd}
\end{equation*}
where all the summands $H^{2r+2k+1-2i}(B, \mathbb Z)$ are trivial. Then the proof follows from Lemma \ref{polarized}.
\end{proof}
\begin{remark}
Proposition \ref{oddcohomology} applies to all known examples of roofs where $n-r$ is odd. Indeed, to the authors' knowledge, in all the known roofs the bases $B$ and $\wt B$ are rational homogeneous varieties and their cohomology is generated by Schubert classes.
\end{remark}

\section{K3 surfaces of degree 12 }\label{k3s}
In the following section we will describe the only non homogeneous roof construction in the list of \cite[Section 5.2.1]{kanemitsu}. Such roof provides pairs of K3 surfaces which are derived equivalent by Corollary \ref{derivedequivalenceK3}.
\subsection{Homogeneous vector bundles on the five dimensional quadric}
Let $Q\subset \PP^6$ be a quadric hypersurface of dimension five, let $\Sc$ be its rank 4 spinor bundle. Ottaviani constructed a 7-dimensional moduli space of rank 3 bundles $\Gc$ such that 
\begin{equation}\label{ott}
    0\arw \Oc \arw \Sc^\vee \arw \Gc \arw 0.
\end{equation}
More precisely, there exists a moduli space isomorphic to $\PP^7\setminus Q_6$ of rank 3 vector bundles $\Gc$ with Chern class $c(\Gc) = (2,2,2)$, and those bundles are the ones satisfying Equation \ref{ott} \cite[Theorem 3.2]{ottaviani}.\\
By \cite[Theorem 3.6 (ii)]{ottaviani} one has $\dim H^0(Q, \Gc(1))=41$ and by the above we have $c(\Gc(1)) = (5,9,12)$. Hence, a section $s$ in such 41-dimensional vector space defines a K3 surface of degree 12 in $\PP^6$.\\
\\
Let $\OO$ denote the complexified Cayley octonions (for details see \cite[Definition 2.4]{kanemitsuottaviani} and the source therein). It is known by \cite[Theorem 2.6]{kanemitsuottaviani} that the projectivization of the Ottaviani bundle can be described in the following way:
\begin{equation*}\label{octonions}
\PP(\Gc^\vee) = \{(x, y)\in \PP(\operatorname{Im}\OO)\times\PP(\operatorname{Im}\OO)\,\, |\,\, x\cdot x=y\cdot y = x\cdot y = 0\}:= X
\end{equation*}
This variety has two natural projections to the quadrics $Q= \{x\in \PP(\operatorname{Im}\OO)| x\cdot x= 0\}$ and $\wt Q= \{y\in \PP(\operatorname{Im}\OO)| y\cdot y= 0\}$ leading to the following diagram:
\begin{equation}\label{cayleyroof}
\begin{tikzcd}
	&		&	X\ar[swap]{ddl}{\pi}\ar{ddr}{\wt\pi}	&		&	\\
	&		&		&		&	\\
Y\ar[hook]{r}	&	Q	&		&	\wt Q\ar[hookleftarrow]{r}&\wt Y
\end{tikzcd}
\end{equation}
where both $Y$ and $\wt Y$ are K3 surfaces described as zero loci of sections of (twisted) Ottaviani bundles $\Gc(1)$ and $\wt\Gc(1)$, and  $\PP(\Gc^\vee(-1))\simeq\PP(\wt\Gc^\vee(-1))\simeq X$. 
\begin{remark}
Diagram \ref{cayleyroof} appears as the roof of type $G_2^\dagger$ in the list of \cite[Section 5.2.1]{kanemitsu}, and it is the only non homogeneous example of such construction. However, one observes that both the quadrics  are homogeneous. What fails to be homogeneous is the Fano 7-fold $X$, which, in contrast with the other examples, is not a generalized flag. On the other hand one can also consider the $G_2$ flag, which admits two projective bundle structures: one to the five dimensional quadric, and a second one being a surjection to the $G_2$-Grassmannian. This construction yields the roof of type $G_2$ studied in \cite{imou,Kuznetsov}, which gives derived equivalent but non isomorphic Calabi--Yau threefolds.
\end{remark}

\begin{subsection}{The roof of type \texorpdfstring{$D_4$}{D4} and its degeneration}\label{degeneration}
Recall the homogeneous roof of type $D_4$. 
\begin{equation}\label{D4roof}
\begin{tikzcd}
	&		&	X_{D_4}\ar[swap]{ddl}{\pi}\ar{ddr}{\wt\pi}	&		&	\\
	&		&		&		&	\\
Y\ar[hook]{r}	&	Q_6	&		&	\wt Q_6\ar[hookleftarrow]{r}&\wt Y
\end{tikzcd}
\end{equation}
Here $Q_6$ and $\wt{Q}_6$ are six-dimensional quadrics representing spinor varieties $OG(4,8)_\pm$ and $X_{D_4}=OG(3,8)=\mathbb{P}_{Q_6}(\Sc(-1))=
\mathbb{P}_{\wt Q_6}(\wt \Sc(-1))$ where $\Sc$ and $\wt\Sc$ are the spinor bundles respectively on $Q_6$ and $\wt Q_6$. Note that $OG(3, 8)$ admits two different projective bundle structures given by the maps $\pi$ and $\wt\pi$, which can be interpreted as the projections determined by the embeddings of the parabolic subgroup of $OG(3, 8)$ inside the parabolic subgroups defining $OG(4, 8)_\pm$.\\
There exists the following short exact sequence on $Q_6$ \cite[Section 3]{ottaviani}, whose restriction to $Q_5$ is a twist of the Equation \ref{ott}:
\begin{equation*}
    0\arw \Oc(1) \arw \Sc^\vee(1) \arw \Gc(1) \arw 0.
\end{equation*}

Note that when the Mukai pair moves in a moduli we get a family of roofs. In this way one can also obtain degenerations of roofs which involve bundles which are not necessarily stable. This is the case for instance in the following context. 
\subsubsection{Degeneration of roofs}
Considering the family of extensions between $\Oc(1)$ and $\Gc(1)$ we see the trivial extension $\Oc(1)\oplus \Gc(1)$ as a degeneration of $\Sc^\vee(1)$.
It follows that $X_{D_4}$ admits a degeneration to $
\hat X_{D_4}=\mathbb{P}_{Q_6}(\Oc(-1)\oplus \Gc^\vee(-1))$. The latter variety is not a roof, but it admits a natural surjection to $Q_6$. A general hyperplane section of $\hat X_{D_4}$ now  gives rise to a K3 surface obtained as a zero locus of a section of $\Oc(1)\oplus \Gc(1)$ on such quadric. In consequence the K3 is given as the zero locus of the restriction of the corresponding section of $\Gc(1)$ to a five dimensional quadric $Q_5$ obtained as a hyperplane section of $Q_6$. If we now consider the restriction of $\Gc(1)$ to the zero locus of a section of $\Oc(1)$ we obtain the roof $G_2^{\dagger}$. The latter roof is a subvariety of some degeneration of the roof of type $D_4$. Moreover, the K3 surfaces associated to this roof are degenerations of K3 surfaces associated to $X_{D_4}$.  We however see in Subsection \ref{sub completeness} that a general K3 surface of degree 12 appears also in the degenerate description.

\end{subsection}
\subsection{Completeness of the family}\label{sub completeness}
In the remainder of this section we prove that the family of K3 surfaces described as sections of an Ottaviani bundle $\Gc(1)$ represents a dense open subset of the family of polarized K3 surfaces of degree 12. In particular, the very general element of this family has Picard number one. We then prove that pairs of K3 surfaces associated to a very general section of the roof $G_2^{\dagger}$ are in general not isomorphic.\\
\\
It is well-known \cite[Corollary 0.3]{mukaigenus} that a general polarized K3 surface of degree 12 has an embedding (defined by its polarization) in the projective space $\PP^7$. If we can prove that our degree 12 K3 surfaces in $\PP^6$ form a 26-dimensional family up to automorphisms of $\PP^6$, then our family can be recovered from the complete 19-dimensional family in $\PP^7$ by means of a projection from one point and hence is also complete.\\
\\
In the dimension count it will be helpful to provide some additional constructions of the studied K3 surfaces. For that, let us denote by $\Cc$ the Cayley bundle defined by the following exact sequence \cite{ottavianicayley}:
\begin{equation}\label{cayleysequence}
        0\arw\Cc(2)\arw\Gc(1)\arw\Oc(2)\arw 0.
    \end{equation}

Before we pass to specific constructions and related dimension counts, let us prepare some ground.
In particular, a large part of what follows will be based on the computation of cohomology of vector bundles on the five dimensional quadric and their restrictions to subvarieties described by the vanishing of sections. We collect all cohomology computations in the following lemma.
\begin{lemma}\label{bottvanishings}
    Let $Q$ be a five dimensional quadric and let $\Gc$ and $\wt\Gc$ be Ottaviani bundles on $Q$. Then, the following holds:
    \begin{enumerate}
        \item $H^k(Q, \Gc^\vee(1)) = 0$ for $k>0$, $H^0(Q, \Gc^\vee(1)) = \CC$
        \item $H^k(Q, \Gc\otimes\wt\Gc(-3)) = 0$ for $k\neq 2$ and $H^2(Q, \Gc\otimes\wt\Gc(-3)) = \CC$.
    \end{enumerate}
    Moreover, let $\Cc$ be the Cayley bundle associated to $\Gc$. Then $H^k(Q, \Cc^\vee\otimes\Cc(-2)) = 0$ for $k\neq 2$, $H^2(Q, \Cc^\vee\otimes\Cc(-2)) = \CC$.
\end{lemma}
\begin{proof}
    To prove each of the statements we use Equations $\ref{ott}$ and $\ref{cayleysequence}$ to reduce the problem to computing cohomology of either twists of $\Gc$ and $\Cc$ (already present in \cite{ottaviani, ottavianicayley}) or twists of symmetric powers of the spinor bundle. In the last case, the Borel--Weil bott theorem provides a simple algorithm to compute the cohomology of such bundles given the associated weight. This process can be easily automatized, see for instance the script \cite{pythonscript}.\\
    The first claim can be proved in the following way: by dualizing the sequence of Equation \ref{cayleysequence} and twisting by $\Oc(2)$ we find the following:
    $$
    0\arw \Oc\arw\Gc^\vee(1)\arw \Cc(1)\arw 0
    $$
    where we used the isomorphism $\Cc^\vee\simeq\Cc(1)$. Then, by \cite[Theorem 3.1]{ottavianicayley} the last term has no cohomology which by the associated long exact sequence of cohomolgy proves our claim.\\
    For the second claim, let us take the tensor product of the sequence of Equation \ref{ott} by $\wt\Gc(-3)$. The first term is acyclic by \cite[Theorem 3.6]{ottaviani}, while the second one can be resolved by the tensor product of the sequence of Equation $\ref{ott}$ (applied to $\wt\Gc$) and $\Sc^\vee(-3)$. In this last sequence the cohomology of the first two terms can be computed by the Borel--Weil--Bott theorem: the first one $\Sc^\vee(-3)$ is acyclic, while for the second one we use the decomposition $\Sc^\vee\otimes \Sc^\vee(-3)\simeq \wedge^2\Sc^\vee(-3)\oplus\sym^2\Sc^\vee(-3)$: the second direct summand has cohomology concentrated in degree 2 with dimension 1 while the first one is acyclic. This last statement follows by the fact that $\wedge^2\Sc^\vee(-3)\simeq \wedge^2\Sc(-1)$ and by the $\Oc(-1)$-twist of the second wedge power of the following sequence \cite[Theorem 2.8]{ottaviani}:
    \begin{equation*}\label{tautologicalspinors}
        0\arw\Sc\arw\Oc^{\oplus 8}\arw\Sc^\vee\arw 0.
    \end{equation*}
    The last claim follows by a similar strategy: we first tensor the sequence of Equation \ref{cayleysequence} by $\Cc^\vee(-4)$, obtaining:
    \begin{equation*}\label{sequencewithcayleys}
        0\arw\Cc^\vee\otimes\Cc(-2)\arw\Cc^\vee\otimes\Gc(-3)\arw\Cc^\vee(-2)\arw 0.
    \end{equation*}
    Then, we observe that $\Cc^\vee(-2)\simeq\Cc(-1)$ is acyclic by \cite[Theorem 3.1]{ottavianicayley} and we compute cohomology of the second term by an appropriate resolution in terms of the sequences of Equations \ref{ott} and \ref{cayleysequence}, in light of the fact that the cohomology of $\Gc(l)$ for every $l$ has been computed in \cite[Theorem 3.6]{ottaviani}. This concludes the proof.
    
\end{proof}

\begin{lemma}\label{onesection}
Let $Q\subset\PP^6$ be a smooth five dimensional quadric hypersurface, let $\Gc$ and $\wt \Gc$ be Ottaviani bundles on $Q$. If $Y=Z(s)=Z(\tilde s)$ for $s\in H^0(Q, \Gc(1))$ and $\wt s\in H^0(Q, \wt \Gc(1))$, then $\Gc=\wt \Gc$ and $s=\wt s$ up to scalar multiplication.
\end{lemma}
\begin{proof}
Under our assumptions we have the following diagram:
\begin{equation*}
\begin{tikzcd}
\cdots\arrow{r}& \Gc^\vee(-1)\arrow{r}{\alpha_{s}}\arrow{d}{\beta} & \mathcal I_Y\arrow{r}\arrow[equal]{d} & 0 \\
\cdots\arrow{r}& \wt\Gc^\vee(-1)\arrow{r}{\alpha_{\widetilde{s}}} & \mathcal I_Y\arrow{r} & 0 
\end{tikzcd}
\end{equation*}
where the rows are given by the Koszul resolutions of $\Ic_Y$ with respect to the two sections. The vertical arrow $\beta$ exists because the map 
$$
\Hom_Q(\Gc^\vee(-1), \wt\Gc^\vee(-1))\longrightarrow \Hom_Q(\Gc^\vee(-1), \mathcal I_Y)
$$
is surjective. This last claim follows from the long cohomology sequence induced by the short exact sequence obtained as the tensor product of $\Gc(1)$ with the Koszul resolution of $\Ic_Y$ with respect to $\wt s$, and by the vanishings $h^1(Q, \Gc\otimes\wedge^2\wt\Gc^\vee(-1)) = h^2(Q, \Gc(-4)) = 0$ which are proven in Lemma \ref{bottvanishings} and \cite[Theorem 3.6]{ottaviani}, once we observe that $\wedge^2\wt\Gc^\vee(-1)\simeq\wt\Gc(-3)$. 
Since Ottaviani bundles are stable \cite[Theorem 3.2]{ottaviani}, the map $\beta$ can be either zero or an isomorphism (cf. \cite[Corollary after Chapter 2 Lemma 1.2.8]{OkonekSpindler}), so we deduce that $s$ and $\wt s$ must be sections of isomorphic Ottaviani bundles. Hence, the proof is completed if we show that $\Hom(\Gc, \wt\Gc)=\CC$. But $\Hom_Q(\Gc, \wt\Gc) = \Hom_Q(\Gc,\Gc)=\CC$, where the first equality follows from the isomorphism $\Gc\simeq \wt \Gc$ and the second again from the stability of $\Gc$, since stable bundles are simple (cf. \cite[Chapter 2, Theorem 1.2.9]{OkonekSpindler}).
\end{proof}
\begin{lemma}\label{pencil}
Let $Y\subset Q$ be a K3 surface satisfying the hypotheses of Lemma $\ref{onesection}$. Then $Y$ is contained in a pencil of quadrics in $\PP^6$.
\end{lemma}
\begin{proof}
The proof follows from observing that $h^0(Q, \Ic_{Y|Q}(2))=1$, where $\Ic_{Y|Q}$ is the ideal sheaf of $Y$ in $Q$. By the Koszul resolution of $\Ic_{Y|Q}$ and the relation $\Gc^\vee \simeq \wedge^2\Gc(-2)$ we find the following exact sequence:
$$
0\arw \Oc(-3)\arw\Gc(-2)\arw\Gc^\vee(1)\arw\Ic_{Y|Q}(2)\arw 0.
$$
Here the first two terms are acyclic and $h^0(Q, \Gc^\vee(1)) = 1$ by Lemma \ref{bottvanishings}, and this proves the claim.

\end{proof}

\begin{lemma}\label{cayleyzerolocus}
    If a K3 surface $Y\subset Q$ is a zero locus of a general section $\sigma$ of $\Gc(1)$, then the intersection $T$ of quadrics containing $Y$ is smooth. Moreover, if $Y$ is a K3 surface which is a zero locus of any section $\sigma$ of $\Gc(1)$ for which the intersection $T$ of quadrics containing it is smooth, then one has $Y = Z(s)\subset T$, where $s\in H^0(T, \Cc|_T(2))$ is a regular section.
\end{lemma}
    \begin{proof}
        Equation \ref{cayleysequence} induces a surjective morphism $f: H^0(Q, \Gc(1))\arw H^0(Q, \Oc(2))$ at the level of cohomology \cite[Theorem 3.1]{ottavianicayley}. 
        The zero locus of $f(\sigma)$ is a variety which is the intersection of $Q$ with another quadric which contains $Y$. Note that this together with Lemma \ref{pencil} implies that $T = Z(f(\sigma))$. Since $f$ is surjective and $\sigma$ is general, it follows that $f(\sigma)$ is also general and hence $T$ is smooth.
        
        Let now $\sigma$ be such that $f(
        \sigma)$ is a regular section (i.e. $T$ is smooth). We consider the restriction of the sequence of Equation \ref{cayleysequence} to $T$. It yields a sequence on cohomology $$0\to H^0(T, \Cc|_T(2))\to H^0(T, \Gc|_T(1))\xrightarrow{f|_T} H^0(T, \Oc_T(2)).$$ 
        Since we observed that $f|_T(\sigma|_T)$ is identically zero, there exists a section $s\in H^0(T, \Cc|_T(2))$ mapping to $\sigma|_T$ via the map $\Cc|_T(2)\to \Gc|_T(1)$. The latter is an embedding hence we have $Z(s)=Y$ and $s$ is a regular section of $\Cc|_T(2)$.
    \end{proof}
    
    \begin{remark} As a consequence of Lemma \ref{pencil} and Lemma \ref{cayleyzerolocus}, we can describe a general K3 surface in the considered family as a  zero locus of a section of $\Cc|_T(2)$ where $T$ is a smooth complete intersection of two five dimensional quadrics. Note that this description also includes K3 surfaces which are not zero loci of sections of $\Gc(1)$, but appear as intersections of the zero locus of $\Cc(2)$ with a quadric, hence are zero loci of sections of the trivial extension $\Cc(2)\oplus\Oc(2)$ (cf. Equation \ref{cayleysequence}).
    \end{remark}

    \begin{proposition}\label{one quadric ottaviani}
        Let $Y$ be a K3 surface which is a zero locus of a general section $\sigma$ of an Ottaviani bundle $\Gc(1)$ on a smooth five dimensional quadric $Q$. Let $\wt Q\subset \PP^6$ be another smooth five dimensional quadric containing $Y$. Then, $Y$ is not a zero locus of any Ottaviani bundle on $\wt Q$.
    \end{proposition}
    \begin{proof}
        The proof, by contradiction, is articulated in two claims. Suppose $Y$ is a zero locus of a general section of $\Gc(1)$ on $Q$ and any section $\wt\Gc(1)$ on $\wt Q$. Then, by Lemma \ref{pencil} and the first part of Lemma \ref{cayleyzerolocus} we see that  $T := Q\cap\wt Q$ is the intersection of quadrics containing $Y$ and is smooth. On the other hand, by the second part of Lemma \ref{cayleyzerolocus} we find that $Y$ is obtained as a zero locus of both $\Cc|_T(2)$ and $\wt\Cc|_T(2)$, where $\Cc$ and $\wt\Cc$ are the Cayley bundles associated to $\Gc(1)$ and $\wt\Gc(1)$ respectively.
        The first claim is that in this case $\Cc|_T(2) = \wt\Cc|_T(2)$. The second claim is that Cayley bundles on two different smooth five dimensional quadrics in $\PP^6$ cannot restrict to the same bundle on the intersection of such quadrics, and this gives the contradiction. \\
        \\
        \textbf{Claim 1.} \emph{Let $Y$ be the zero locus of a section $s\in H^0(T, \Cc|_T(2))$  and in the same time of another section $\wt s\in H^0(T, \wt \Cc|_T(2))$.  Then $\Cc|_T = \wt\Cc|_T$, and $s=\tilde s$ up to scalar multiplication.}\\
        
        The proof is similar to the one of Lemma \ref{pencil}, hence we will be brief. The section $s$ determines a surjection $\alpha_s:\Cc^\vee|_T(-2)\arw\Ic_{Y|T}$, while $\wt s$ defines a second surjection $\alpha_{\wt s}:\wt\Cc^\vee|_T(-2)\arw\Ic_{Y|T}$. Let us prove that $\Hom_T(\Cc^\vee|_T(-2), \wt\Cc^\vee|_T(-2))\arw\Hom_T(\Cc^\vee(-2), \Ic_{Y|T})$ is surjective: one has a short exact sequence
        $$
        0\arw\Cc|_T(-1)\arw\Cc\otimes\wt\Cc^\vee|_T\arw\Cc|_T(2)\otimes\Ic_{Y|T}\arw 0
        $$
        obtained by tensoring the Koszul resolution of $\Ic_{Y|T}$ by $\Cc(2)$. The surjectivity of the map above follows from $h^1(T, \Cc|_T(-1)) = 0$ which is a consequence of the vanishings $H^\bullet(Q, \Cc(-3))$\linebreak $= H^\bullet(Q, \Cc(-1)) = 0$ \cite[Theorem 3.1]{ottavianicayley}, and the Koszul resolution of $T$ as a subvariety of $Q$. This, in turn, implies the existence of a nontrivial morphism $\beta:\Cc^\vee|_T(-2)\arw\wt\Cc^\vee|_T(-2)$. Restrictions of Cayley bundles to $T$ are stable: in fact, a destabilizing subobject $L$ would have rank 1 and slope $\mu(L)\geq -2$, where we use the standard definition $\mu(L) := \deg c_1(L)/\operatorname{rk}(L)$. This means that $L$ is a line bundle on $T$, which by Lefschetz Theorem has the form $\Oc_T(l)$ for some integer $l$, and has slope $4l$. This would imply an injection $\Oc_T \xhookrightarrow{\hspace{10pt}}\Cc|_T(-l)$. But $h^0(Q, \Cc|_T(-l))=0$ from the Koszul resolution of $T$ and the vanishings $h^0(Q, \Cc(-l)) = h^1(Q, \Cc(-2-l) = 0$ for any $l \geq -1/2$ by \cite[Theorem 3.1]{ottavianicayley}. Therefore $\beta$ being nontrivial must be an isomorphism, thus $s$ and $\wt s$ must be sections of isomorphic bundles. The proof is concluded by the observation that $\Hom_T(\Cc|_T, \Cc|_T)=\CC$, which follows from stability of $\Cc_T$.
        \\
        
        \textbf{Claim 2.} \emph{Let $\Cc$ and $\wt\Cc$ denote Cayley bundles of two five dimensional quadrics $Q$ and $\wt Q$ in $\PP^6$ with smooth complete intersection $T:=Q\cap\wt Q$. Then $\Cc|_T\not \simeq\wt\Cc|_T$}.\\

        Let us proceed by contradiction: assume that $\Cc|_T \simeq\wt\Cc|_T$.
        Observe that $\Cc(-2)$ is acyclic by \cite[Theorem 3.1]{ottavianicayley}. Hence, from the exact sequence
        $$0\to \Cc(-2)\to \Cc\to \Cc|_T \to 0,$$
        we obtain that the restriction map $H^1(Q, \Cc)\to H^1(T, \Cc|_T) $ is an isomorphism.  In particular,  we have $h^1(T, \Cc|_T) = h^1(Q, \Cc) = 1$ and the unique non-trivial extension representing the generator of $\Ext^1_T (\Oc_T(2), \Cc|_T(2))$ is the restriction of the unique non-trivial extension representing the generator of $\Ext^1_Q (\Oc(2), \Cc(2))$ which is
        $$
        0\arw \Cc(2)\arw \Gc(1)\arw \Oc(2)\arw 0
        $$
        (see Equation \ref{cayleysequence}). We conclude that the generator of $\Ext^1_T (\Oc_T(2), \Cc|_T(2))$ is represented by
        $$
        0\arw \Cc|_T(2)\arw \Gc|_T(1)\arw \Oc_T(2)\arw 0.
        $$
        Similarly, the generator of $\Ext^1_T (\Oc_T(2), \wt \Cc|_T(2))$ is represented by the sequence
        $$
        0\arw \wt\Cc|_T(2)\arw \wt\Gc|_T(1)\arw \Oc_T(2)\arw 0.
        $$
        Now, since $\Cc|_T(2)\simeq \wt \Cc|_T(2)$, we obtain $\Gc|_T(1)\simeq \wt \Gc|_T(1)$.
        
        Using the same argument with the sequence 
       $$0\to \Gc^\vee(-2)\to \Gc^\vee\to \Gc^\vee|_T \to 0,$$
       and the one dimensional extension $\Ext^1_Q(\Oc(1), \Gc^\vee(1))$ represented by
       $$
       0\arw \Gc^\vee(1)\arw\Sc(1)\arw\Oc(1)\arw 0,
       $$
       we finally obtain $\Sc|_T\simeq \wt\Sc|_T$, where $\Sc$ and $\wt \Sc$ are spinor bundles on $Q$ and $\wt Q$ respectively. Indeed, the cohomologies of $\Gc^\vee$ and $\Gc^\vee(-2)$ are computed by the dual of Equation \ref{ott} and its twist by $\Oc(-2)$. In both the sequences the calculations of cohomology of the remaining terms is either trivial or it follows by previous computations (proof of Lemma \ref{bottvanishings}).
       
       Let us now consider a general hyperplane section  $Z=T\cap H$. Since $\Sc|_T\simeq \wt\Sc|_T$, then also $\Sc|_Z\simeq \wt\Sc|_Z$. Note that, the variety $Z$ is smooth and isomorphic to the intersection  $G\cap \wt G$ with $G:=Q\cap H$ and $\wt G:=\wt Q \cap H$, two distinct four dimensional quadrics interpreted as two translates of the Grassmannian $G(2, 4)$ in its Pl\"ucker embedding in $\PP^5=H$. If we denote by $\Uc$, $\Qc$, $\wt \Uc$ and $\wt \Qc$  the tautological and quotient bundles of the two Grassmannians respectively, we have decompositions $$\Sc|_Z \simeq \Uc|_Z\oplus\Qc^\vee|_Z \text{ and }\wt\Sc|_Z \simeq \wt\Uc|_Z\oplus\wt\Qc^\vee|_Z$$ cf. \cite[Theorem 1.4 and Example 1.5]{ottaviani}.
       
       Observe furthermore that all bundles $\Uc|_Z$, $\Qc|_Z$, $\wt \Uc|_Z$ and $\wt \Qc|_Z$ are stable. Indeed, let us illustrate the proof for $\Uc|_Z$, as the others are identical. The argument is essentially the same we used in the proof of Claim 1. The bundle $\Uc|_Z$ has rank 2 and determinant $\Oc_Z(-1)$, hence a destabilizing subobject is a line bundle $\Oc_Z(l)$ with $l\geq -1/2$, which would determine an injective map \linebreak $\Oc\xhookrightarrow{\hspace{10pt}}\Uc|_Z(-l)$. Such map cannot exist because by the following sequence:
       $$
       0\arw\Uc(-2-l)\arw\Uc(-l)\arw\Uc|_Z(-l)\arw 0
       $$
       one sees that $\Uc|_Z(-l)$ has no sections for any $l\geq-1/2$, hence proving stability.
       
       Therefore, the isomorphisms
       $$
       \Uc|_Z\oplus\Qc^\vee|_Z\simeq \Sc|_Z\simeq\wt\Sc_Z\simeq\wt\Uc|_Z\oplus\wt\Qc^\vee|_Z
       $$
       imply that either $\Uc|_Z\simeq\wt\Uc|_Z$ or $\Uc|_Z\simeq\wt\Qc^{\vee}|_Z.$ This is equivalent to: $\Uc^{\vee}|_Z\simeq\wt\Uc^{\vee}|_Z$ or $\Uc^{\vee}|_Z\simeq\wt\Qc|_Z.$ However it is a standard result on the geometry of Fano threefolds of index 2 and degree 4 that all these bundles are in general pairwise distinct (see for example \cite[Section 5.1]{sashainstantons}). We provide an explicit argument here for completeness.
       
       Observe that since $\Uc(-2)$ is acyclic, by the long exact sequence of cohomology associated to
       $$
       0\arw\Uc^\vee(-2)\arw\Uc^\vee\arw\Uc^\vee|_Z\arw 0
       $$
       we find that the restriction map on the space of global sections is an isomorphism
       \begin{equation}\label{isomorphism_sections}
            H^0(G, \Uc^\vee)\simeq H^0(Z, \Uc^\vee|_Z),
       \end{equation}
       and the same argument holds for $\Qc$. Consider a general section $s$ of $\Uc^\vee|_Z$ and $\hat s$ the corresponding section of $\Uc^\vee$ given by the isomorphism of Equation \ref{isomorphism_sections}. Recall that the zero locus of $\hat s$ is a plane contained in $G$. Therefore: $$
       Z(s) = Z\cap Z(\hat s) = \wt G\cap Z(\hat s)
       $$
       is a smooth conic, by generality assumptions and the fact that $\Uc^\vee$ is globally generated. In particular, $Z(\hat s)$ is the unique plane containing $Z(s)$. 
       We hence can associate to the bundle $\mathcal U^{\vee}|_Z$ a unique family $\Pc(\Uc^\vee|_Z)$ of planes in $\mathbb{P}^5$ determined by its global sections. Furthermore, since
       $$
       \bigcup_{\hat{s}\in H^0(G,\mathcal U^{\vee})} Z(\hat s)=G,
       $$
       the family $\Pc(\Uc^\vee|_Z)$ determines $G$. In fact, $\Pc(\Uc^\vee|_Z)$ represents one of the two connected components of the family of planes in $G$.
       
       In the same way, the bundle $\Qc|_Z$ defines the family $\Pc(\Qc|_Z)$ corresponding to the second component of the family of planes in $G$ and hence it also determines $G$. On the other hand the bundles $\wt \Uc^{\vee}|_Z$, $\wt \Qc|_Z$ determine in the same way $\Pc(\wt \Uc^{\vee}|_Z)$ and $\Pc({\wt \Qc}|_Z)$, the two components of the family of planes in the quadric $\wt G$. 
       
       The bundle $\Uc^\vee|_Z$ can be isomorphic to $\wt\Uc^\vee|_Z$ or $\wt\Qc|_Z$ only if $\Pc(\Uc^\vee|_Z)$ coincides with either $\Pc(\wt \Uc^{\vee}|_Z)$ or $\Pc(\wt \Qc|_Z)$. In each case we get $G=\wt G$ which gives a contradiction.

           
    \end{proof}



\begin{proposition}\label{picard}
The very general K3 surface described as a zero locus of a section of $\Gc(1)$, where $\Gc$ is an Ottaviani bundle, has Picard number one.
\end{proposition}
\begin{proof}
In Proposition \ref{one quadric ottaviani} we proved that each of the studied K3 surfaces, although it is contained in a pencil of quadrics, appears as a zero locus of a section of an Ottaviani bundle only on one quadric in the pencil. Now a K3 surface $Y\subset \mathbb{P}^6$ is determined by the data of a quadric $Q\subset \mathbb{P}^6$ an Ottaviani bundle $\Gc$ on $Q$ and a section of $\Gc(1)$ up to a constant multiple. Furthermore, by the discussion above, if there exists an automorphism of $\mathbb{P}^6$ mapping one K3 surface in the family to another one then it will need to map all the corresponding data to the other. We can hence perform the following dimension count.

The space of sections of an Ottaviani bundle on $Q$ has dimension 41, and the moduli space of Ottaviani bundles on $Q$ is 7-dimensional.
Since the action of $\Aut{Q}=\operatorname{Spin}(7)$ is transitive on the moduli space of Ottaviani bundles, and a K3 surface $Y\subset Q$ determines the section, the (projective) dimension of the family is given by:
$$
40-21+7=26
$$
where $21-7$ is the dimension of the space of automorphisms of $Q$ fixing an Ottaviani bundle.
Hence,  the family we are describing is a 26 dimensional family (of classes up to automorphisms of $\PP^6$) of embedded K3 surfaces of degree 12 in $\PP^6$. Since each K3 of degree 12 has a projective embedding in $\PP^7$ a complete family of K3 of degree 12 in $\PP^6$ can be described by a $19+7=26$-parameter space, via projection from a point in $\PP^7$. This proves that our family is complete, therefore the very general element has Picard number one.
\end{proof}


\subsection{Roofs of quadrics and non isomorphic K3 surfaces}
Let us consider a roof $X$ where the bases of its vector bundles are smooth quadrics $Q$ and $\wt Q$, and the associated Calabi--Yau pairs have dimension two. Both $Q$ and $\wt Q$ have cohomology generated by algebraic classes, hence we have an isometry $(\wt j_{*} \circ \wt q^{*}|_{T_{\wt Y}})^{-1}\circ j_{*} \circ q^{*}|_{T_Y}$ between transcendental lattices  $T_Y\simeq T_{\wt Y}$ for each associated Calabi--Yau pair $(Y, \wt Y)$ (in fact K3 pair in this case). Moreover, if $Y$ and $\wt Y$ are general,  by \cite[proof of Lemma 4.1]{oguiso} 
$T_Y$ and $T_{\wt Y}$ admit no self-isometries different from $\pm \operatorname{Id}$. Hence, to prove that $Y$ and $\wt Y$ are not isomorphic it is enough to prove that none of the isometries $\pm (\wt j_{*} \circ \wt q^{*}|_{T_{\wt Y}})^{-1}\circ j_{*} \circ q^{*}|_{T_Y}$  extends to an isometry between $H^2(Y,\mathbb Z)$ and $H^2(\wt Y,\mathbb Z)$. \\
\\
For that we will consider the following notation. Given a lattice $R$, let us call $dR$ the discriminant group defined by the exact sequence
\begin{equation*}\label{discses}
0\arw R\arw \Hom_\ZZ(R, \ZZ)\arw dR\arw 0.
\end{equation*}
Recall that if $Y$ is a K3 surface of Picard number one, with $L$ denoting the generator of its Picard group, then  $\langle L\rangle=(T_Y)^{\perp}\subset H^2(Y,\mathbb Z)$ and since $H^2(Y,\mathbb Z)$ is unimodular  $dT_Y\simeq d\langle L \rangle\simeq  \mathbb Z/L^2\mathbb Z$ and there is a distinguished generator of $dT_Y$ corresponding to $[\frac{L}{L^2}]$ under the canonical identification $dT_Y\simeq d\langle L\rangle$. Similarly, $\wt Y$ is a K3 surface of Picard number one and if we denote by $\wt L$ the generator of its Picard group, we have $[\frac{\wt L}{\wt L^2}]$ representing the generator of $dT_{\wt Y}$ associated to the embedding  $T_{\wt Y}\subset H^2(\wt Y,\mathbb Z)$.\\
\\
Note that under the generality assumption for $(Y, \wt Y)$ (using \cite[proof of Lemma 4.1]{oguiso}) each of the lattices $T_{Y}$, $T_{\wt Y}$ can be identified with $T_M$ in a unique way up to $\pm \operatorname{Id}$, and this identification is given by  $\pm j_{*} \circ q^{*}|_{T_Y}$ and $\pm \wt j_{*} \circ \wt q^{*}|_{T_{\wt Y}}$. Furthermore, $H^k(M, \mathbb Z)$ is unimodular and hence $dT_{Y}$ and $dT_{\wt Y}$ admit canonical identifications with $dH^{2r}_{alg}(M, \ZZ)$.  On the other hand, by Theorem \ref{integralcohomology} and Lemma \ref{polarized} both $H^2(Y,\mathbb Z)$ and $H^2(\wt Y,\mathbb Z)$ admit Hodge isometric embeddings into $H^{2r}(M, \mathbb Z)$ extending the embeddings of the transcendental lattices.  We conclude that under our identifications $[\frac{L}{L^2}]=\pm[\frac{ j_{*}q^{*} L}{L^2}]\in dH^{2r}_{alg}(M, \ZZ)$ and $[\frac{\wt L}{\wt L^2}]=\pm[\frac{ \wt j_{*}\wt q^{*} \wt L}{\wt L^2}]\in dH^{2r}_{alg}(M, \ZZ)$. To prove that $Y$ and $\wt Y$ are not isomorphic it remains to check that $[\frac{ j_{*}q^{*} L}{L^2}]$ and $\pm [\frac{\wt j_{*} \wt q^{*} \wt L}{\wt L^2}]$ are distinct elements in $dH^{2r}_{alg}(M, \ZZ)$. Indeed, if $Y$ and $\wt Y$ were isomorphic then the isomorphism would need to map $L$ to $\wt L$ and then $[\frac{L}{L^2}]=\pm[\frac{\wt L}{\wt L^2}]$ in $dT_M=dH^{2r}_{alg}(M, \ZZ)$.

This is checked in each of the two known cases by the following Lemma.

\begin{lemma}\label{seven}
Let $X$ be a roof of type $G_2^\dagger$ or $D_4$,  $M\subset X$ a very general hyperplane and $Y, \wt Y$ the associated pair of K3 surfaces of degree $12$. Then there is a unique isometry of transcendental lattices $T_Y\simeq T_{\wt Y}$ up to $\pm \operatorname{Id}$ and this isometry  descends to an isomorphism of discriminant groups which maps $[\frac{L}{12}]$ to $\pm 7[\frac{\wt L}{12}]$.
\end{lemma}
\begin{proof}
By the discussion before the Lemma we just need to compare $[\frac{ j_{*}q^{*} L}{12}]$ and $[\frac{\wt j_{*} \wt q^{*} \wt L}{12}]$ in $dH^{2r}_{alg}(M, \ZZ)$.  For that we will present $j_{*}q^{*} L$ and $\wt j_{*} \wt q^{*} \wt L$ in a chosen basis of $H^{2r}_{alg}(M, \ZZ)$. The computations differ slightly in each of the two cases. Let us first illustrate the proof for the roof of type $G_2^\dagger$, which is easier because of the simpler structure of the cohomology ring of the quadric, which is odd dimensional.\\
By abuse of notation let us denote by $L\in H^2(Q,\ZZ)$ the hyperplane class of $Q$, its restriction to $Y$ which is the polarization, as well as its pullback to $X$ together with its restriction to $M$. Fix $\xi\in H^2(X, \ZZ)$ the class of the Grothendieck line bundle $\Oc_{\PP(\Gc^\vee(-1))}(1)$, we also denote by $\xi$ its restriction to $M$.\\
\\
\textbf{Claim:} We claim that a basis for $H^{6}_{alg}(M, \ZZ)$ is given by the classes $\Pi, L^2\xi, L\xi^2$, where $\Pi=\frac{1}{2} L^3$ is the class of a plane in $Q$. 

To see that, first observe that these are generators $H^{6}_{alg}(X, \ZZ)$ which after restriction to $M$ define a sublattice of $H^{6}_{alg}(M, \ZZ)$. 

Now, given the Grothendieck relation on $X$:
\begin{equation}\label{groth}
\xi^3-5L\xi^2+9L^2\xi-12\Pi=0
\end{equation}
we can write the intersection form:
\begin{center}
 \begin{tabular}{c | c c c}\label{intersectionott}
       & $\Pi$ & $L^2\xi$ & $L\xi^2$\\ [1ex] 
 \hline 
 $\Pi$ & 0 & 1 & 5\\[1ex] 
 $L^2\xi$ & 1 & 10 & 32\\[1ex] 
 $L\xi^{2}$ & 5 & 32 & 82\\
\end{tabular}
\end{center}
It follows that the sublattice $\langle\Pi, L^2\xi, L\xi^2\rangle\subset H^{6}_{alg}(M, \ZZ)$ has rank 3 and discriminant $-12$, which by Theorem \ref{integralcohomology}, Lemma \ref{polarized} and Equation \ref{perp trans} is also the  rank and discriminant of  $H^{6}_{alg}(M, \ZZ)$. We conclude that  $\langle\Pi, L^2\xi, L\xi^2\rangle$ is the whole $H^{6}_{alg}(M, \ZZ)$ proving the claim.\\
\\
Knowing that $(j_{*}q^{*} L)\cdot L^2\xi=(j_{*}q^{*} L)\cdot \Pi=0$ and  $j_{*}q^{*} L$ is an effective primitive class in $H^{6}_{alg}(M, \ZZ)$  we get:
\begin{equation}
    j_{*}q^{*} L=L{\xi}^2-5L^2\xi+18 \Pi
\end{equation}
which, from the relation 
$$
\xi=L+\wt L
$$
gives $$j_{*}q^{*} L=7\wt L{\xi}^2-23\wt L^2\xi+42 \wt\Pi.$$
Now by the same argument repeated for $\wt Y$ we have 
$$\wt j_{*}\wt q^{*} \wt L=\wt L{\xi}^2-5\wt L^2\xi+18 \wt \Pi.$$
We conclude that
$$\frac{1}{12}(j_{*}q^{*} L - 7 \wt j_{*}\wt q^{*} \wt L)= \wt L^2\xi-7\wt \Pi \in H^{6}_{alg}(M, \ZZ).$$\\
\\
Let us now focus on the roof of type $D_4$. Here the K3 surfaces are zero loci of $\Sc^\vee(1)$. The cohomology ring of a six dimensional quadric is slightly more complicated, since there exist two disjoint families of maximal isotropic linear spaces $\Pi_1, \Pi_2$. They satisfy the following relations in the cohomology ring:
\begin{equation}\label{eq_coh_relations_D4}
L^3=\Pi_1+\Pi_2,\,\,\, \Pi_1\cdot L=\Pi_2\cdot L,\,\,\, \Pi_1^2=\Pi_2^2=0,\,\,\, \Pi_1\cdot \Pi_2 = 1.
\end{equation}
By the same argument as above, we can construct a basis of the middle cohomology $H^8(M, \ZZ)$ given by the classes $\Pi_1 L, \Pi_1\xi, \Pi_2 \xi, L^2\xi^2, L \xi^3$.
The Grothendieck relation is
\begin{equation}\label{eq_grothendieckD4}
    \xi^4-6L\xi^3+14L^2\xi^2-14\Pi_1\xi-16\Pi_2\xi+12\Pi_1L=0
\end{equation}
which yields the following intersection matrix:
\begin{center}
 \begin{tabular}{c | c c c c c}
       & $\Pi_1 L$ & $\Pi_1\xi$ & $\Pi_2\xi$ & $L^2\xi^2$ & $L\xi^3$\\ [1ex] 
 \hline 
 $\Pi_1 L$ & 0 & 0 & 0 & 1 & 6\\[1ex] 
 $\Pi_1\xi$ & 0 & 0 & 1 & 6 & 22\\[1ex] 
 $\Pi_2\xi$ & 0 & 1 & 0 & 6 & 22\\
 $L^2\xi^2$ & 1 & 6 & 6 & 44 & 126\\
 $L\xi^3$ & 6 & 22 & 22 & 126 & 308
\end{tabular}
\end{center}
As for the non homogeneous roof, we can compute the representation of $j_* q^* L$ in terms of the basis above:
\begin{equation}\label{eq_jqL_D4_1}
    \begin{split}
        j_* q^* L &= L\xi^3 -6L^2\xi^2 +14\Pi_1\xi + 14\Pi_2\xi - 30\Pi_1 L\\
                  &=L\xi^3 -6L^2\xi^2 +14 L^3\xi - 15 L^4
    \end{split}
\end{equation}
where the second equality follows from the relations of Equation \ref{eq_coh_relations_D4}. By substituting the expression $L = \xi-\wt L$ in Equation \ref{eq_jqL_D4_1} we find
\begin{equation*}
    j_*q^*L = -6\xi^4+29\wt L\xi^3-54\wt L^2\xi^2+46\wt L^3\xi-15\wt L^4
\end{equation*}
which by Equation \ref{eq_grothendieckD4} can be rewritten as
\begin{equation*}
    \begin{split}
        j_*q^*L = -7\wt L\xi^3+30\wt L^2\xi^2-38\wt\Pi_1\xi - 50\wt\Pi_2\xi + 42\wt\Pi_1\wt L.
    \end{split}
\end{equation*}
On the other hand, we can apply the same argument which leads to Equation \ref{eq_jqL_D4_1} in order to get:
\begin{equation*}
    \wt j_*\wt q^*\wt L = \wt L\xi^3 -6\wt L^2\xi^2 +14\wt \Pi_1\xi + 14\wt \Pi_2\xi - 30\wt \Pi_1 \wt L.
\end{equation*}
Finally, we have:
\begin{equation*}
    \frac{1}{12}(j_*q^*L+7\wt j_*\wt q^*\wt L) = -\wt L^2\xi^2 + 5\wt \Pi_1\xi + 4\wt\Pi_2\xi - 14\wt\Pi_1\wt L\in H_{alg}^8(M, \ZZ).
\end{equation*}
\end{proof}
As a result of the discussion above, we get the following:
\begin{corollary} Let $X$ be a roof of type $G_2^\dagger$ or $D_4$,  $M\subset X$ a very general hyperplane and $(Y, \wt Y)$ the associated pair of K3 surfaces of degree $12$. Then $Y$ and $\wt Y$ are not isomorphic.
\end{corollary}

\subsection{Fourier--Mukai transform}\label{kernels}
Let us consider a pair $Y$, $\wt Y$ of K3 surfaces. Then, by the derived Torelli theorem, they are derived equivalent if and only if there exists a Hodge isometry of their Mukai lattices \cite[Theorem 4.2.1]{orlovderivedtorelli}.\\
Let us now specialize to a general pair of K3 surfaces associated to a roof of type $G_2^\dagger$. Then they are derived equivalent by Theorem \ref{derivedequivalenceK3} and,  by the derived Torelli theorem, it is possible to find an explicit expression of the Mukai vector of the associated cohomological Fourier--Mukai transform. This, in turn, allows us to gain some information on the Fourier--Mukai transform defining the equivalence.\\
\\
Let $\wt H(Y, \ZZ)$ be the Mukai lattice of $Y$. Then there exists a Hodge isometry
\begin{equation}\label{theta}
\theta: \wt H(Y, \ZZ)\arw H^6_{alg}(M, \ZZ)
\end{equation}
which can be explicitly described in terms of the basis $\{\Pi, L^2\xi, L\xi^2\}$ of  $H^6_{alg}(M, \ZZ)$. In particular, we must have $\theta(L)=j_*q^*L=18\Pi-5L^2\xi+L\xi^2$. The images $\theta(v)$ and $\theta(w)$ of the generators $v$ of $H^0(Y, \ZZ)$ and $w$ of $H^4(Y, \ZZ)$, can be determined, up to an overall sign and up to exchanging them, by the conditions $\theta(v)\cdot\theta(L)=\theta(w)\cdot \theta(L)=\theta(v)\cdot \theta(v) = \theta(w)\cdot \theta(w)=0$ and  $\theta(v)\cdot \theta(w)=1$. Imposing such conditions we get (up to exchanging $v$ with $w$ or an overall sign) $\theta(v)= \Pi$ and $\theta(w)= -5\Pi+ L^2\xi$. Note that, a priori, unicity of $\theta$ is not obvious.\\
\\
Let us now consider the derived equivalence $\Phi:\dbcoh(Y)\arw\dbcoh(\wt Y)$ given by the isometry $T_Y\simeq T_{\wt Y}$ discussed in Corollary \ref{derivedequivalenceK3}. Then, by \cite[Theorem 4.2.4, Theorem 4.2.1]{orlovderivedtorelli}, it induces an isometry $\iota_\Phi:\wt H(Y, \ZZ)\arw \wt H(\wt Y, \ZZ)$, such that the image $\iota_\Phi(v)$ under this isometry is the Mukai vector defining the Fourier--Mukai transform. Summing all up, we have the following commutative diagram:
\begin{equation}\label{FMclasses}
    \begin{tikzcd}
    H^6_{alg}(M, \ZZ)\ar{r}{\operatorname{id}}   &       H^6_{alg}(M, \ZZ)\\
    \wt H(Y, \ZZ)\ar{u}{\theta}\ar{r}{\iota_\Phi}     &       \wt H(\wt Y, \ZZ)\ar{u}{\wt\theta}
    \end{tikzcd}
\end{equation}
where the vertical arrows are given by Equation \ref{theta}. Then, we can find explicitly the image $\iota_\Phi(v)$ of the generator $v$ of $H^0(Y, \ZZ)$ using $\iota_\Phi(v)=\wt\theta^{-1}(\theta(v))$. By direct computation we find the Mukai vector $ (2,1,-3)$. This Mukai vector satisfies the assumption that the coefficient corresponding to the component $H^0(Y,\mathbb{Z})$ is $>1$, hence the argument from the proof of \cite[Theorem 4.2.3]{orlovderivedtorelli} applies and gives the following.
\begin{proposition}
Let $Y$, $\wt Y$ be a pair of K3 surfaces of Picard number 1 defined by a hyperplane section $M$ of a roof of type $G_2^\dagger$. Then $\wt Y$ is isomorphic to the moduli space $\Mc_Y(u)$ of stable sheaves $\Fc$ on $Y$ with Mukai vector
$$
v(\Fc) = u = (2,1,-3).
$$
\end{proposition}
\begin{proof}
The proof follows from the computation above and the proof of \cite[Theorem 4.2.3]{orlovderivedtorelli}. Indeed, in the reference the author identifies by means of the Torelli theorem the variety $\wt Y$ with the Moduli space $\Mc_Y(u)$   of stable sheaves on $Y$ with given Mukai vector $u$, which is also a K3 surface by \cite{Mukai87}.
\end{proof}

 We thus recover the well-known Fourier--Mukai transform yielding Mukai duality for K3 surfaces of degree 12 \cite[Example 1.3]{MukaiK3}. This also gives an alternative proof of non-isomorphicity of  $Y$ and $\wt Y$.
\begin{remark}
It is tempting to extend this approach to the roof of type $D_4$. However, instead of the isometries $\theta$ and $\wt\theta$, one can construct isometries of $H^8_{alg}(M, \ZZ)$ with a lattice of rank 5 containing a hyperbolic lattice and the Picard lattice. This construction  is highly non unique, and it is not known, a priori, if a diagram such as Diagram \ref{FMclasses} exists.
\end{remark}

\section{\texorpdfstring{$D$}{D}-brane categories}\label{branes}
Let us consider a pair $(Y, \wt Y)$ of derived equivalent Calabi--Yau varieties associated to a roof $X$. By an argument based on Kn\"orrer periodicity and Landau--Ginzburg models, we show that the derived equivalence $\dbcoh(Y)\simeq \dbcoh(\wt Y)$ lifts to an equivalence of matrix factorization categories. Let us first recall some definitions, while for the general theory we refer to \cite{rennemosegal}, \cite{shipman}.
\begin{definition}
We call Landau--Ginzburg model the data of:
\begin{enumerate}
    \item A stack $\Xc = [V/G]$ where $V$ is a smooth quasi-projective variety endowed with the action of a reductive group $G$ and an $R$-charge $\CC^*_R\simeq \CC^*$
    \item A function $w:V\mapsto \CC$ called \emph{superpotential}, which is $G$-invariant and has weight $2$ with respect to the $R$-charge action
    \item $-1\in \CC_R^*$ acts trivially on $\Xc$
\end{enumerate}
\end{definition}
\begin{definition}
A graded $D$-brane on a Landau--Ginzburg model $(\Xc, w, G, \CC_R^*)$ is a $\CC^*_R$-equivariant vector bundle $\Fc$ endowed with an endomorphism $d_\Fc$ of $\CC_R^*$-weight $1$ such that $d_\Fc^2=w\cdot \operatorname{Id}_\Fc$.
\end{definition}
One can define a morphism of $\CC_R^*$-weight 1
$$
\begin{tikzcd}[row sep = tiny, /tikz/column 1/.append style={anchor=base east} ,/tikz/column 2/.append style={anchor=base west}]
\sHom(\Fc, \Gc)\ar{rr}{d}&& \sHom(\Fc, \Gc)\\
\phi\ar[maps to]{rr} && d_\Gc\circ \phi - \phi\circ d_\Fc
\end{tikzcd}
$$ 
which has the property $d^2=0$. Then we can view $(\sHom(\Fc, \Gc), d)$ as a complex graded by the $\CC_R^*$ charge, and construct a dg-category $\operatorname{MF}(\Xc, w)$, from which one can define a triangulated category $\operatorname{DMF}(\Xc, w)$ as a Verdier quotient with respect to a suitable subcategory of acyclic objects. There exists a rich literature on this topic, the construction of $\operatorname{DMF}(\Xc, w)$ has been carried out in full detail, for example, in \cite{shipman}.
\\
\\
Let us specialize to the case $\Xc\simeq \Ec^\vee$, where $(\Ec, B)$ is a Mukai pair such that $\Ec$ is a $G$-homogeneous vector bundle on a smooth $G$-homogeneous variety $B$. Then, given a regular section $s\in H^0(B, \Ec)$, a natural choice for a superpotential is the function
\begin{equation}\label{superpotential}
    \begin{tikzcd}[row sep = tiny, /tikz/column 1/.append style={anchor=base east} ,/tikz/column 2/.append style={anchor=base west}]
        V \ar{rr}{w} & & \CC  \\
        (b,v)\ar[maps to]{rr} & & v\cdot s(b)
    \end{tikzcd}
\end{equation}
This function is $G$-invariant by construction, and it is always possible to define a $\CC^*$-action such that it has weight 2, so that it fulfills the requirements of the definition of a $R$-charge.\\
In this setting, there exists a result called Kn\"orrer periodicity (see, for example, \cite[Theorem 3.4]{shipman}) where an equivalence between the derived category of the zero locus $Y=Z(s)$ and the derived category of matrix factorizations $\operatorname{DMF}(\Xc, w)$ has been constructed:
\begin{theorem}[Kn\"orrer periodicity]
Let $(\Xc, w, G, \CC_R^*)$ be a Landau--Ginzburg model and $\pi:\Ec\arw B$ a vector bundle over a smooth variety, such that $\Xc\simeq \Ec^\vee$. Let $Y$ be the zero locus of a section $s$ of  $\Ec$. Assume that $s$ is regular and hence $Y$ is smooth of codimension $r=\operatorname{rk} \Ec$. Let $p:\pi^{-1}(Y)\arw Y$ and $i:\pi^{-1}(Y)\hookrightarrow \Xc$. Then the functor:
$$
i_* p^*: \dbcoh(Y)\arw \operatorname{DMF}(\Xc, w)
$$
is an equivalence of categories.
\end{theorem}
Let us now consider a roof $X\simeq \PP(\Ec^\vee)\simeq \PP(\wt \Ec^\vee)$, where the vector bundles $\Ec$ and $\wt\Ec$ are respectively $G$- and $\wt G$-homogeneous. Then, if we call $\Xc:= \Ec^\vee$, $\wt \Xc:= \wt \Ec^\vee$, fixing a section $ S \in H^0(X, \Lc)$ we can construct two Landau--Ginzburg models $(\Xc, w, G, \CC_R^*)$ and $(\wt \Xc, \wt w, \wt G, \wt \CC_R^*)$ where the superpotentials are defined as in Equation \ref{superpotential} by the pushforwards of $ S $ to $B$ and $\wt B$. Then, if $(Y, \wt Y)$ is a derived equivalent Calabi--Yau pair defined by $ S $, we establish the following diagram, where all arrows are equivalences:
\begin{equation*}
    \begin{tikzcd}
    \operatorname{DMF}(\Xc, w) & & \operatorname{DMF}(\wt\Xc, \wt w) \\
    & & \\
    \dbcoh(Y)\ar{uu}\ar{rr} & & \dbcoh(\wt Y)\ar{uu}
    \end{tikzcd}
\end{equation*}
Here the vertical arrows are given by Kn\"orrer periodicity. \\
For the roof of type $A_4$, for every hyperplane section the authors constructed two Landau--Ginzburg models as above, related by an explicit phase transitions described in terms of variation of GIT with respect to the action of a non Abelian group \cite[Section 6]{kr}. In this context, the fact that the derived equivalence $\dbcoh(Y)\simeq \dbcoh(\wt Y)$ lifts to an equivalence of matrix factorization categories is physically motivated by the fact that $D$-brane categories of different phases of the same gauged linear sigma model are expected to be equivalent, and such categories of branes are mathematically described with the language of matrix factorizations. It would be an interesting problem to establish a similar picture for other derived equivalent Calabi--Yau pairs arising from roofs.

\vskip2cm
\noindent
\author{Micha\l{} Kapustka:}
\address{Institute of Mathematics of the Polish academy of Sciences ul. Śniadeckich 8, 00-656 Warszawa, Poland.\\
University of  Stavanger, Department of Mathematics and Physics, NO-4036 Stavanger, Norway }
\email{michal.kapustka@uis.no}
\\
\\
\author{Marco Rampazzo:}
\address{
University of Bologna, Department of Mathematics, Piazza di Porta San Donato 5, Bologna, Italy.
}
\email{marco.rampazzo3@unibo.it}


\begin{thebibliography}{}


\bibitem[BCP20]{bcp}
Borisov, Lev A.; C\u{a}ld\u{a}raru, Andrei; Perry, Alexander. Intersections of two Grassmannians in $\PP^9$. J. Reine Angew. Math. 760 (2020), 133–162.

\bibitem[BFM21]{BFM}
Bernardara, Marcello; Fatighenti, Enrico; Manivel, Laurent. Nested varieties of K3 type. J. Éc. polytech. Math. 8 (2021), 733-778.

\bibitem[BO02]{bondalorlovdk}
Bondal, Alexei; Orlov, Dmitri O. Derived categories of coherent sheaves. Proceedings of the International Congress of Mathematicians, Vol. II (Beijing, 2002), Higher Ed. Press, Beijing (2002), 47–56


\bibitem[HL18]{hl}
Hassett, Brendan; Lai, Kuan-Wen.
Cremona transformations and derived equivalences of K3 surfaces. Compos. Math. 154 (2018), no. 7, 1508–1533.

\bibitem[IMOU19]{imou}
Ito, Atsushi; Miura, Makoto; Okawa, Shinnosuke; Ueda, Kazushi. The class of the affine line is a zero divisor in the Grothendieck ring: via G2-Grassmannians. J. Algebraic Geom. 28 (2019), no. 2, 245–250.

\bibitem[IMOU20]{imouk3}
Ito, Atsushi; Miura, Makoto; Okawa, Shinnosuke; Ueda, Kazushi.
Derived equivalence and Grothendieck ring of varieties: the case of K3 surfaces of degree 12 and abelian varieties. Selecta Math. (N.S.) 26 (2020), no. 3, Paper No. 38

\bibitem[Kan19]{kanemitsuottaviani}
Kanemitsu, Akihiro. Extremal rays and nefness of tangent bundles. Michigan Math. J. 68 (2019), no. 2, 301–322.

\bibitem[Kan18]{kanemitsu}
Kanemitsu, Akihiro. Mukai pairs and simple $K$ -equivalence. e-Print: arXiv:1812.05392, 2018

\bibitem[Kaw17]{kawamatadk}
Kawamata, Yujiro. Birational geometry and derived categories. Surveys in differential geometry 2017. Celebrating the 50th anniversary of the Journal of Differential Geometry, Surv. Differ. Geom., 22, Int. Press, Somerville, MA, 2018, 291–317.

\bibitem[Kuz12]{sashainstantons}
Kuznetsov, Alexander. Instanton bundles on Fano threefolds. Cent. Eur. J. Math. 10(4) (2012), 1198-1231

\bibitem[Kuz18]{Kuznetsov}
Kuznetsov, Alexander. Derived equivalence of Ito-Miura-Okawa-Ueda Calabi-Yau 3-folds. J. Math. Soc. Japan 70 (2018), no. 3, 1007–1013.

\bibitem[KR19]{kr}
Kapustka, {Micha\l}; Rampazzo, Marco. Torelli problem for Calabi-Yau threefolds with GLSM description. Commun. Number Theory Phys. 13 (2019), no. 4, 725–761. 

\bibitem[KS18]{ks}
Kuznetsov, Alexander; Shinder, Evgeny. Grothendieck ring of varieties, D- and L-equivalence, and families of quadrics. Selecta Math. (N.S.) 24 (2018), no. 4, 3475–3500.

\bibitem[LL03]{LarsenLunts}
Larsen, Michael; Lunts, Valery A. Motivic measures and stable birational geometry. Mosc. Math. J. 3 (2003), no. 1, 85–95.

\bibitem[LMS75]{LascuMumfordScott}
Lascu, Alexandru T.; Mumford, David; Scott, David B.
The self-intersection formula and the ``formule-clef''. 
Math. Proc. Cambridge Philos. Soc. 78 (1975), 117–123. 

\bibitem[Mor18]{morimuraflop}
Morimura, Hayato. Derived equivalence for Mukai flop via mutation of semiorthogonal decomposition. e-Print: arXiv:1812.06413. (2018).

\bibitem[Muk81]{MukaiFourier}
Mukai, Shigeru. Duality between $D(X)$ and $D(\hat X)$ with its application to Picard sheaves. Nagoya Math. J. 81 (1981), 153-175 .

\bibitem[Muk87]{mukaigenus}
Mukai, Shigeru. Curves, K3 surfaces and Fano 3-folds of genus $\leq 10$. Algebraic geometry and commutative algebra, Vol. I, Kinokuniya, Tokyo, (1988), 357–377. 

\bibitem[Muk87b]{Mukai87}
Mukai, Shigeru. On the moduli space of bundles on K3 surfaces. I. Vector bundles on algebraic varieties (Bombay, 1984), Tata Inst. Fund. Res. Stud. Math., 11, Tata Inst. Fund. Res., Bombay (1987), 341–413.

\bibitem[Muk98]{MukaiK3}
Mukai, Shigeru. Duality of polarized K3 surfaces. New trends in algebraic geometry (Warwick, 1996),
London Math. Soc. Lecture Note Ser., 264 (1999), Cambridge Univ. Press, Cambridge, 311–326.

\bibitem[Ogu02]{oguiso}
Oguiso, Keiji. K3 surfaces via almost-primes. Math. Res. Lett. 9 (2002), no. 1, 47–63. 

\bibitem[Orl97]{orlov}
Orlov, Dmitri O. Equivalences of derived categories and K3 surfaces. J. Math. Sci. 84 (1997), 1361-1381.

\bibitem[Orl03]{orlovderivedtorelli}
Orlov, Dmitri O. Derived categories of coherent sheaves and equivalences between them.  Uspekhi Mat. Nauk 58 (2003), no. 3(351), 89–172; translation in Russian Math. Surveys 58 (2003), no. 3, 511–591


\bibitem[OR18]{or}
Ottem, John C.; Rennemo, J\o rgen V. A counterexample to the birational Torelli problem for Calabi--Yau threefolds. J. Lond. Math. Soc. (2) 97 (2018), no. 3, 427–440. 

\bibitem[OSS80]{OkonekSpindler}
Okonek, Christian; Schneider, Michael; Spindler, Heinz. Vector bundles on complex projective spaces. Birkhauser (1980).

\bibitem[Ott88]{ottaviani}
Ottaviani, Giorgio. Spinor bundles on quadrics.
Trans. Amer. Math. Soc. 307 (1988), no. 1, 301–316.

\bibitem[Ott90]{ottavianicayley}
Ottaviani, Giorgio. On Cayley bundles on the five-dimensional quadric. Boll. Un. Mat. Ital. A (7) 4 (1990), no. 1, 87–100. 

\bibitem[R21]{pythonscript}
Rampazzo, Marco. A Python script to compute cohomology of irreducible homogeneous vector bundles on rational homogeneous varieties. GitHub repository available at the link https://github.com/marcorampazzo/bott-theorem.

\bibitem[RS19]{rennemosegal}
Rennemo, J\o rgen V.; Segal, Ed. Hori--mological projective duality. Duke Math. J., 168 (2019), no. 11, 2127-2205.

\bibitem[Rød00]{rodland}
 R\o dland, Einar A. The Pfaffian Calabi-Yau, its mirror, and their link to the Grassmannian G(2,7). Compositio Math. 122 (2000), no. 2, 135–149.

\bibitem[Shi12]{shipman}
Shipman, Ian. A geometric approach to Orlov's theorem. Compos. Math. 148 (2012), no. 5, 1365–1389.

\bibitem[Ued19]{uedaflop}
Ueda, Kazushi. $G_2$-Grassmannians and derived equivalences. Manuscripta Math. 159 (2019), no. 3-4, 549–559. 


\bibitem[Voi02]{voisin}
Voisin, Claire. Hodge theory and complex algebraic geometry. Cambridge University Press (2002).

\end{thebibliography}
\end{document}